\documentclass{article}%
\usepackage{amsfonts}
\usepackage{amsmath}
\usepackage[nomarginpar]{geometry}%
\usepackage{amssymb}%
\usepackage{graphicx}
\providecommand{\U}[1]{\protect\rule{.1in}{.1in}}
\newtheorem{theorem}{Theorem}

\newtheorem{corollary}[theorem]{Corollary}

\newtheorem{lemma}[theorem]{Lemma}

\newtheorem{proposition}[theorem]{Proposition}

\newenvironment{proof}[1][Proof]{\noindent\textbf{#1.} }{\ \rule{0.5em}{0.5em}}
\begin{document}

\title{On augmented eccentric connectivity index of graphs and trees}
\author{Jelena Sedlar\\\textit{University of Split, Faculty of civil engeneering, architecture and
geodesy,}\\\textit{Matice hrvatske 15, 21000 Split, Croatia}}
\maketitle

\begin{abstract}
In this paper we establish all extremal graphs with respect to augmented
eccentric connectivity index among all (simple connected) graphs, among trees
and among trees with perfect matching. For graphs that turn out to be extremal
explicit formulas for the value of augmented eccentric connectivity index are derived.

\end{abstract}

\section{Introduction}

Several topological indices based on graph theoretical notion of eccentricity
have been recently proposed and/or used in QSAR and QSPR studies. Namely,
eccentric connectivity index (\cite{bib_Sharma(1997)ECI}), eccentric distance
sum (\cite{bib_Gupta(2002)}), adjacent eccentric distance sum
(\cite{bib_Sardana(2003)}) and augmented and super augmented eccentric
connectivity index (\cite{bib_Bajaj(2005)}, \cite{bib_Bajaj(2006)},
\cite{bib_Dureja(2008)} and \cite{bib_Dureja(2009)}). These indices have been
shown to be very useful (predicting pharmaceutical properties), therefore
their mathematical properties have been studied too. The most extensive study
has been conducted for eccentric connectivity index, for which extremal graphs
and trees have been established (\cite{bib_DoslicVuk(2010)},
\cite{bib_Zhou(2010)}, \cite{bib_Ilic(2011)Chem}). Furthermore, the eccentric
connectivity index of some special kinds of graphs was studied such as
unicyclic graphs and different kinds of hexagonal systems
(\cite{bib_Doslic(2011)nanoECI},\cite{bib_Doslic(2011)Chains}). For a detailed
survey on these and other results concerning eccentric connectivity index we
refer the reader to \cite{bib_IlicSurvey(2010)}. Recently, mathematical
properties of eccentric distance sum started to be investigated too. There are
some results on eccentric distance sum of trees and unicyclic graphs
(\cite{bib_Ilic(2011)EDS}) and of general graphs
(\cite{bib_Ilic(2011)EDSgraphs}). As for the augmented eccentric connectivity
index, there are some results with explicit formulas for several classes of
graphs, in particular for some open and closed unbranched polymers and
nanostructure (\cite{bib_Doslic(2012)AECI}). Otherwise, augmented eccentricity
index was not very much studied.

In this paper we present the results concerning extremal graphs and values of
augmented eccentric connectivity index on class of simple connected graphs, on
trees and on trees with perfect matching. The paper is organized as follows.
In the second section 'Preliminaries' some basic notions and also the notation
are introduced. Also, explicit formulas for the value of augmented eccentric
connectivity index for some specific graphs (such as paths, stars, etc.) which
will later be proved as extremal are derived. Third section is named 'Extremal
trees'. In it we establish all minimal and extremal trees with respect to
augmented eccentric connectivity index. Interestingly, it turns out that
maximal tree generally is not a star as is the case with other eccentricity
based indices, but a specific kind of tree with diameter $4.$ In fourth
section we establish extremal trees among trees with perfect matching.
Finally, in fifth section we use the results for trees to establish the
extremal graphs in class of general simple connected graphs.

\section{Preliminaries}

In this paper we consider only simple connected graphs. We will use the
following notation: $G$ for graph, $V(G)$ or just $V$ for its set of vertices,
$E(G)$ or just $E$ for its set of edges. With $n$ we will denote number of
vertices in graph $G.$ For two vertices $u,v\in V$ we define \textit{distance}
$d(u,v)$ of $u$ and $v$ as the length of shortest path connecting $u$ and $v.$
Given the notion of distance we can define several other notions based on
distance. First, for a vertex $u\in V$ we define \textit{eccentricity}
$\varepsilon(u)$ as the maximum of $d(u,v)$ over all $v\in V.$ Furthermore, we
define \textit{diameter} $D$ of graph $G$ as the maximum of $d(u,v)$ over all
pairs of vertices $u,v\in V.$ A path $P$ in $G$ connecting vertices $u$ and
$v$ is called \textit{diametric} if $d(u,v)=D.$ The set of all vertices with
minimum eccentricity in $G$ is called \textit{center} of $G$ and such vertices
are called \textit{central}. For a vertex $u\in V$ a \textit{degree} $\deg(u)$
is defined as number of vertices from $V $ adjacent to $u$. Now, we can define
\textit{augmented eccentric connectivity index} of a graph $G$ as
\[
\xi^{ac}(G)=\sum_{u\in V}\frac{M(u)}{\varepsilon(u)}%
\]
where $M(u)$ is product of degrees of all neighbors of $u$ and $\varepsilon
(u)$ is eccentricity of $u.$ Sometimes, for brevity sake, this index will be
called 'augmented ECI'.

\bigskip

Let us now define some special kinds of graphs. First, $K_{n}$ will denote a
\textit{complete graph} on $n$ vertices. Special class of graphs which will be
of interest are trees. A \textit{tree} is a graph with no cycles. It is easily
seen that tree has only one central vertex if $D$ is even, and two central
vertices if $D$ is odd. We say that a vertex in tree $T$ is a leaf if its
degree is $1,$ otherwise we say that a vertex is non-leaf. Also, we say that a
vertex in a tree is \textit{branching} if its degree is greater or equal than
$3$. We say that a tree $T$ is \textit{spanning tree} of graph $G$ if
$V(T)=V(G)$ and $E(T)\subseteq E(G).$ Now, $P_{n}$ will denote a \textit{path}
on $n$ vertices and $S_{n}$ will denote a \textit{star} on $n$ vertices. We
will also specially consider trees of diameter $4.$ Let us therefore introduce
some interesting classes of graphs with diameter $4.$ We say that a tree $T$
is \textit{degree balanced} if its diameter is $4$ and all neighbors of (the
only) central vertex differ in degree by at most one. With $TB_{n,k}$ we will
denote degree balanced tree on $n$ vertices with degree of central vertex
being $k.$ Note that there is only one such tree up to isomorphism. Now, from
definition follows that neighbors of central vertex in degree balanced tree
$T$ can have only two degrees, say $p-1$ and $p.$ Note that $p$ is determined
by $k$ and holds%
\[
p=\left\lceil \frac{n-1}{k}\right\rceil .
\]
But, on the other hand $k$ is not determined by $p.$ We have
\[
\frac{n-1}{p}\leq k<\frac{n-1}{p-1}.
\]

Having this in view, we define (almost) perfect degree balance in a tree.
Namely, we say that a degree balance is \textit{perfect} if all neighbors of
central vertex have the degree $p$, we say that balance is \textit{almost
perfect} if maximum possible number of neighbors of central vertex have the
degree $p$. Note that degree balanced tree is in (almost) perfect balance if
and only if
\[
k=\left\lceil \frac{n-1}{p}\right\rceil .
\]
An example of tree with degree balance, almost perfect degree balance and
perfect degree balance is shown in Figure \ref{Figure1}.%

%TCIMACRO{\FRAME{ftbpFU}{5.4578in}{0.8406in}{0pt}{\Qcb{Trees: $TB_{12,5}$ with
%degree balance, $TB_{14,5}$ with almost perfect degree balance and $TB_{16,5}$
%with perfect degree balance.}}{\Qlb{Figure1}}{figure1.eps}%
%width 5.4578in;  height 0.8406in;  depth 0pt;  original-width 27.1473in;
%original-height 4.0603in;  cropleft "0";  croptop "1";  cropright "1";
%cropbottom "0";  filename 'Figure1.eps';file-properties "XNPEU";}}}%
%BeginExpansion
\begin{figure}[ptb]%
\centering
\includegraphics[
height=0.8406in,
width=5.4578in
]%
{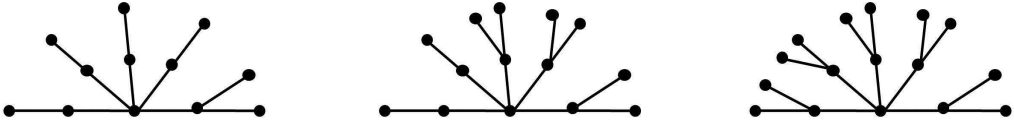}%
\caption{Trees: $TB_{12,5}$ with degree balance, $TB_{14,5}$ with almost
perfect degree balance and $TB_{16,5}$ with perfect degree balance.}%
\label{Figure1}%
\end{figure}
%EndExpansion

Now, we will establish exact values of augmented ECI for some of these graphs,
which will later be proved as extremal for some class of graphs. Let us denote
$H_{n}=\sum_{i=1}^{n}\frac{1}{i}.$ By direct calculation we obtain the
following proposition.

\begin{proposition}
\label{prop_values}For paths $P_{n}$ on $n\geq5$ vertices, stars $S_{n}$ on
$n\geq4$ vertices, degree balanced trees $TB_{n,\left\lceil \frac{n-1}%
{3}\right\rceil }$ on $n\geq8$ vertices, degree balanced graphs $TB_{n,\frac
{n}{2}}$ on $n\geq6$ vertices where $n$ is even, complete graphs $K_{n}$ on
$n\geq2$ vertices holds:

\begin{enumerate}
\item $\xi^{ac}(P_{n})=\left\{
\begin{tabular}
[c]{ll}%
$8(H_{n-1}-H_{\frac{n-2}{2}})-(\frac{4}{n-1}+\frac{4}{n-2}),$ & for $n$
even,\\
$8(H_{n-1}-H_{\frac{n-3}{2}})-(\frac{12}{n-1}+\frac{4}{n-2}),$ & for $n$ odd,
\end{tabular}
\right.  $

\item $\xi^{ac}(S_{n})=1+\frac{(n-1)^{2}}{2},$

\item $\xi^{ac}(TB_{n,\left\lceil \frac{n-1}{3}\right\rceil })=\left\{
\begin{tabular}
[c]{ll}%
$\frac{3^{k}}{2}+\frac{k^{2}}{3}+\frac{3k}{2},$ & for $n=3k+1,$\\
$3^{k-1}+\frac{k^{2}}{3}+\frac{3k}{2}-1$ & for $n=3k,$\\
$2\cdot3^{k-2}+\frac{k^{2}}{3}+\frac{3k}{2}-\frac{1}{2}$ & for $n=3k-1.$%
\end{tabular}
\right.  $

\item $\xi^{ac}(TB_{n,\frac{n}{2}})=2^{\frac{n}{2}-2}+\frac{1}{12}\left(
n^{2}+3n-6\right)  $,

\item $\xi^{ac}(K_{n})=n\cdot\left(  n-1\right)  ^{n-1}.$
\end{enumerate}
\end{proposition}

To conclude, we still need the notion of matching. A \textit{matching} in a
graph $G$ is collection of edges $M$ from $G$ such that no vertex from $G$ is
incident to two edges from $M.$ The \textit{size} of matching is number of
edges it contains. We say that matching $M$ is \textit{perfect} if every
vertex from $G$ is incident to one edge from $M.$ Obviously, only graphs with
even number of vertices can have perfect matching.

\section{Extremal trees}

In this section we want to establish trees with minimum and maximum value of
augmented eccentric connectivity index. First, we will do the minimum. For
that purpose we need the following theorem which gives the transformation of
tree which increases diameter, but decreases the value of augmented ECI.

\begin{theorem}
\label{tm_path}Let $T\not =P_{n}$ be a tree on $n$ vertices and let
$P=v_{0}v_{1}\ldots v_{D}$ be a diametric path in $T$ chosen so that the first
branching vertex is furthest possible from $v_{0}$. Let $v_{i}$ be the first
branching vertex on $P.$ If $D>2$ and $i=1$ and $\deg(v_{i+1})>2$ then let
$u=v_{i+1},$ else let $u=v_{i}.$ Let $w_{1},\ldots,w_{k}$ be $k$ neighbors of
$u$ outside of $P$ ($1\leq k\leq\deg(u)-2$). For tree $T^{\prime}$ obtained
from $T$ by deleting edges $uw_{1},\ldots uw_{k}$ and adding edges $v_{0}%
w_{1},\ldots,v_{0}w_{k}$ holds
\[
\xi^{ac}(T)>\xi^{\alpha c}(T^{\prime}).
\]

\end{theorem}

\begin{proof}
Note that this transformation does not decrease eccentricity of any vertex. On
the other hand, the only vertex whose degree increases is $v_{0}.$ Let us
denote $m_{i}=\deg(v_{i})$ and $m_{w_{i}}=\deg(w_{i}).$ Cases when $D\leq3 $
are easily verified, therefore we distinguish three remaining cases when
$D>3.$ Tree transformations from these cases are illustrated with Figure
\ref{Figure2}.

Case 1: $D>3$ and $u=v_{1}.$

Note that in this case $m_{2}=2$ and $m_{w_{i}}=1$ for every $i.$ Therefore,
the only vertex for which $M(v)$ increases from $T$ to $T^{\prime}$ is
$v_{1}.$ We have
\begin{align*}
\xi^{ac}(T)-\xi^{\alpha c}(T^{\prime})  & \geq\frac{M(v_{1})}{\varepsilon
(v_{1})}-\frac{M^{\prime}(v_{1})}{\varepsilon^{\prime}(v_{1})}+\frac{M(v_{2}%
)}{\varepsilon(v_{2})}-\frac{M^{\prime}(v_{2})}{\varepsilon^{\prime}(v_{2}%
)}+\sum_{i=1}^{k}\left(  \frac{M(w_{i})}{\varepsilon(w_{i})}-\frac{M^{\prime
}(w_{i})}{\varepsilon^{\prime}(w_{i})}\right)  \geq\\
& \geq\frac{2}{D-1}-\frac{2(k+1)}{D-1}+\frac{m_{1}\cdot m_{3}}{D-2}%
-\frac{\left(  m_{1}-k\right)  \cdot m_{3}}{D-2}+k\left(  \frac{m_{1}}%
{D}-\frac{k+1}{D}\right)  =\\
& =\frac{-2k}{D-1}+\frac{m_{3}\cdot k}{D-2}+\frac{k\cdot\left(  m_{1}%
-k-1\right)  }{D}>\left[  m_{3}\geq2,m_{1}\geq k+1\right]  >0
\end{align*}

Case 2: $D>3$ and $u=v_{2}.$

The only vertices for which $M(v)$ possibly increases are $v_{0}$ and $v_{1}$.
We will neutralize increase in $M(v_{0})$ by decrease in $M(v_{2}),$ and also
neutralize increase in $M(v_{1})$ by decrease in $M(v_{3}),M(w_{1}%
),\ldots,M(w_{k})$. Let $m_{w}=m_{w_{1}}\cdot\ldots\cdot m_{w_{k}}$ and
$c_{2}=M(v_{2})/(m_{w}\cdot m_{1}).$ Note that $c_{2}\geq2$ because of
$m_{3}\geq2$ (which follows from $D>3$). We have
\begin{align*}
\Delta_{1}  & =\frac{M(v_{0})}{\varepsilon(v_{0})}-\frac{M^{\prime}(v_{0}%
)}{\varepsilon^{\prime}(v_{0})}+\frac{M(v_{2})}{\varepsilon(v_{2})}%
-\frac{M^{\prime}(v_{2})}{\varepsilon^{\prime}(v_{2})}\geq\frac{m_{1}}%
{D}-\frac{m_{1}\cdot m_{w}}{D}+\frac{m_{1}\cdot m_{w}\cdot c_{2}}{D-2}%
-\frac{m_{1}\cdot c_{2}}{D-2}=\\
& =-\frac{m_{1}(m_{w}-1)}{D}+\frac{m_{1}\cdot c_{2}\cdot(m_{w}-1)}{D-2}\geq0.
\end{align*}
Now, let $c_{3}=M(v_{3})/m_{2}$. Note that all neighbors of $w_{i}$ except
$u=v_{2}$ are of degree $1$ and therefore $M(w_{i})=m_{2}$ for every $i.$ We
have
\begin{align*}
\Delta_{2}  & =\frac{M(v_{1})}{\varepsilon(v_{1})}-\frac{M^{\prime}(v_{1}%
)}{\varepsilon^{\prime}(v_{1})}+\frac{M(v_{3})}{\varepsilon(v_{3})}%
-\frac{M^{\prime}(v_{3})}{\varepsilon^{\prime}(v_{3})}+%
%TCIMACRO{\dsum \limits_{i=1}^{k}}%
%BeginExpansion
{\displaystyle\sum\limits_{i=1}^{k}}
%EndExpansion
\left(  \frac{M(w_{i})}{\varepsilon(w_{i})}-\frac{M^{\prime}(w_{i}%
)}{\varepsilon^{\prime}(w_{i})}\right)  \geq\\
& \geq\frac{m_{2}}{D-1}-\frac{\left(  m_{2}-k\right)  \left(  k+1\right)
}{D-1}+\frac{m_{2}\cdot c_{3}}{\varepsilon(v_{3})}-\frac{(m_{2}-k)\cdot c_{3}%
}{\varepsilon(v_{3})}+%
%TCIMACRO{\dsum \limits_{i=1}^{k}}%
%BeginExpansion
{\displaystyle\sum\limits_{i=1}^{k}}
%EndExpansion
\left(  \frac{m_{2}}{D-1}-\frac{k+1}{D-1}\right)  =\\
& =-\frac{k\left(  m_{2}-k-1\right)  }{D-1}+\frac{k\cdot c_{3}}{\varepsilon
(v_{3})}+\frac{k\left(  m_{2}-k-1\right)  }{D-1}>0
\end{align*}
Therefore,
\[
\xi^{ac}(T)-\xi^{\alpha c}(T^{\prime})\geq\Delta_{1}+\Delta_{2}>0.
\]

Case 3: $D>3$ and $u=v_{i}$ ($i\geq3$).

The only vertices for which $M(v)$ possibly increases are $v_{0}$ and $v_{1}$.
We will neutralize increase in $M(v_{0})$ by decrease in $M(v_{i}),$ and also
neutralize increase in $M(v_{1})$ by decrease in $M(v_{i-1})$. Let
$m_{w}=m_{w_{1}}\cdot\ldots\cdot m_{w_{k}}$ and let $c_{i}=M(v_{i})/\left(
m_{i-1}\cdot m_{w}\right)  $ We have
\begin{align*}
\Delta_{1}  & =\frac{M(v_{0})}{\varepsilon(v_{0})}-\frac{M^{\prime}(v_{0}%
)}{\varepsilon^{\prime}(v_{0})}+\frac{M(v_{i})}{\varepsilon(v_{i})}%
-\frac{M^{\prime}(v_{i})}{\varepsilon^{\prime}(v_{i})}\geq\frac{2}{D}%
-\frac{2\cdot m_{w}}{D}+\frac{m_{i-1}\cdot m_{w}\cdot c_{i}}{\varepsilon
(v_{i})}-\frac{m_{i-1}\cdot c_{i}}{\varepsilon(v_{i})}=\\
& =-\frac{2\left(  m_{w}-1\right)  }{D}+\frac{m_{i-1}\cdot c_{i}\cdot
(m_{w}-1)}{\varepsilon(v_{i})}\overset{m_{i-1}=2}{\geq}0.
\end{align*}
Also, from $i\geq3$ we know that $v_{1}\not =v_{i},$ so we have
\begin{align*}
\Delta_{2}  & =\frac{M(v_{1})}{\varepsilon(v_{1})}-\frac{M^{\prime}(v_{1}%
)}{\varepsilon^{\prime}(v_{1})}+\frac{M(v_{i-1})}{\varepsilon(v_{i-1})}%
-\frac{M^{\prime}(v_{i-1})}{\varepsilon^{\prime}(v_{i-1})}\geq\frac{2}%
{D-1}-\frac{2(k+1)}{D-1}+\frac{2\cdot m_{i}}{\varepsilon(v_{i-1})}%
-\frac{2\left(  m_{i}-k\right)  }{\varepsilon(v_{i-1})}=\\
& =-\frac{2k}{D-1}+\frac{2k}{\varepsilon(v_{i-1})}>\left[  \varepsilon
(v_{i-1})<D-1\right]  >0
\end{align*}
Therefore, we conclude
\[
\xi^{ac}(T)-\xi^{\alpha c}(T^{\prime})\geq\Delta_{1}+\Delta_{2}>0.
\]

\end{proof}

\bigskip%

%TCIMACRO{\FRAME{ftbpFU}{3.9461in}{2.9265in}{0pt}{\Qcb{Tree transformations
%from Cases 1, 2 and 3 of Theorem \ref{tm_path}.}}{\Qlb{Figure2}}%
%{figure2.eps}{\special{ language "Scientific Word";  type "GRAPHIC";
%width 3.9461in;  height 2.9265in;  depth 0pt;  original-width 26.1225in;
%original-height 19.3251in;  cropleft "0";  croptop "1";  cropright "1";
%cropbottom "0";  filename 'Figure2.eps';file-properties "XNPEU";}}}%
%BeginExpansion
\begin{figure}[ptb]%
\centering
\includegraphics[
height=2.9265in,
width=3.9461in
]%
{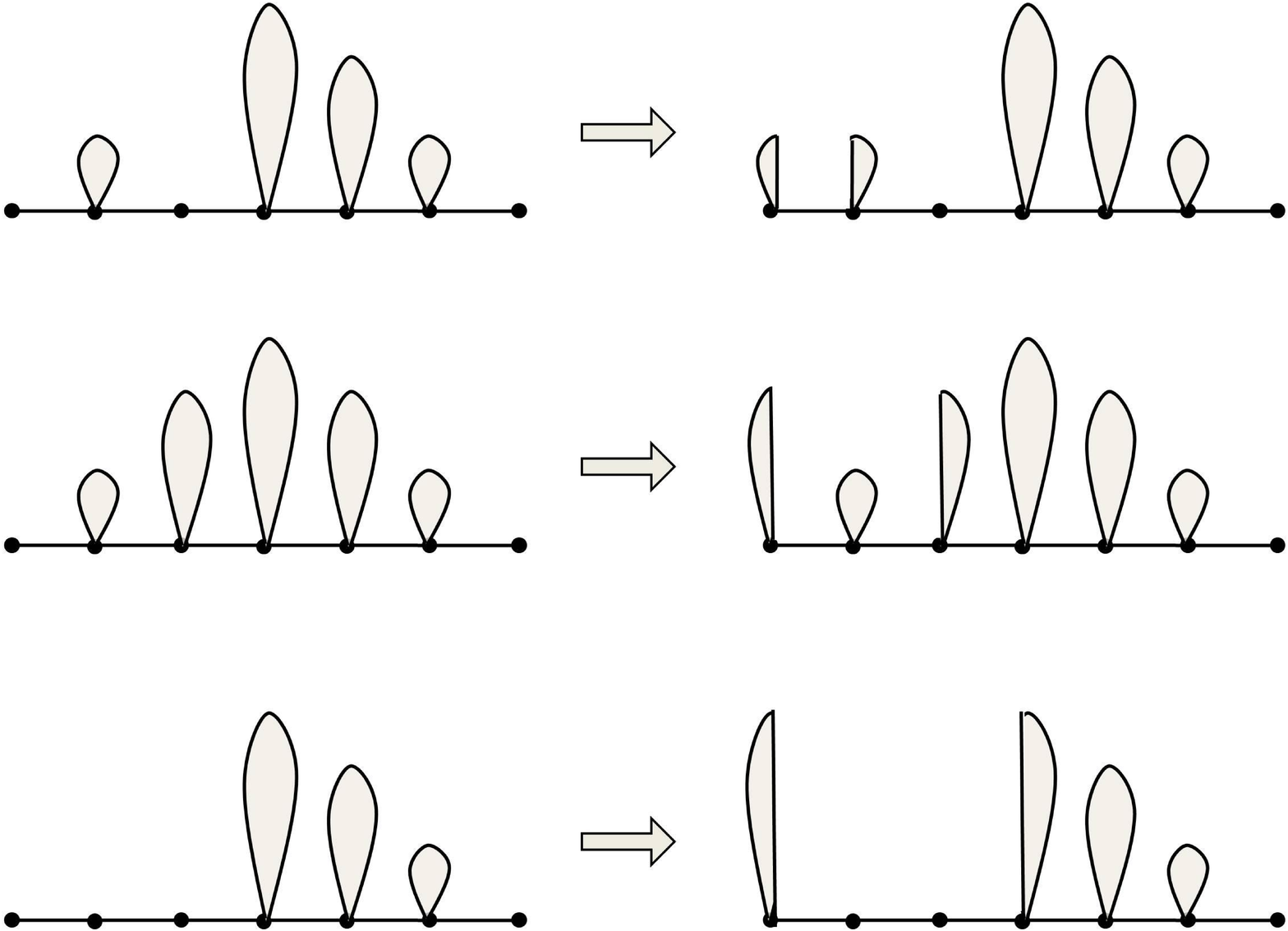}%
\caption{Tree transformations from Cases 1, 2 and 3 of Theorem \ref{tm_path}.}%
\label{Figure2}%
\end{figure}
%EndExpansion

\begin{corollary}
\label{Cor_path}Let $T\not =P_{n}$ be a tree on $n$ vertices. Then
\[
\xi^{ac}(T)>\xi^{ac}(P_{n}).
\]

\end{corollary}

\begin{proof}
Note that transformation of tree from Theorem \ref{tm_path} increases diameter
of the tree. Therefore, applying that transformation consecutively on $T$ we
obtain in the end path $P_{n}$ which by that Theorem has smaller value of
$\xi^{ac}$ then $T.$
\end{proof}

\bigskip

Now that we found a tree with minimum value of augmented ECI, we want to find
a tree with maximum value of augmented ECI. One could expect a star $S_{n}$ to
have maximum value of augmented ECI, as that was the case of other
eccentricity based indices. But, comparing the value of $S_{16}$ and
$TM_{16,5}$ we obtain
\begin{align*}
\xi^{ac}(S_{16})  & =1+15\cdot\frac{15}{2}=\frac{227}{2}=113.5,\\
\xi^{ac}(TB_{16,5})  & =\frac{3^{5}}{2}+5\cdot\frac{5}{3}+10\cdot\frac{3}%
{4}=\frac{412}{3}=137.33,
\end{align*}
which clearly indicates that $S_{n}$ is not maximal tree with respect to value
of augmented ECI. Now, we want to establish which trees are maximal. For that
purpose, we need the following theorem.

\begin{theorem}
\label{tm_maxStar}Let $T$ be a tree on $n$ vertices with diameter $D\geq5$ and
let $P=v_{0}v_{1}\ldots v_{D}$ be a diametric path such that $M(v_{2}%
)/\deg(v_{1})\leq M(v_{D-2})/\deg(v_{D-1})$. Let $w_{1},\ldots,v_{k}$ be all
pendent vertices of $v_{1}.$ For a tree $T^{\prime}$ obtained from $T$ by
deleting edges $v_{1}w_{1},\ldots,v_{1}w_{k}$ and adding vertices
$v_{D-1}w_{1},\ldots,v_{D-1}w_{k}$ holds
\[
\xi^{ac}(T)<\xi^{ac}(T^{\prime}).
\]

\end{theorem}

\begin{proof}
Note that by this transformation eccentricities of vertices do not increase.
The only vertex for which $M(v)$ decreases is $v_{2}.$ We will neutralize this
decrease by increase for $v_{D-1}.$ For the simplicity sake, let $m_{i}%
=\deg(v_{i})$ for $v_{i}\in P.$ Taking into account that $\varepsilon
(v)\geq\varepsilon^{\prime}(v)$ for every $v\in T$ we have
\begin{align*}
\xi^{ac}(T)-\xi^{ac}(T^{\prime})  & \leq\frac{M(v_{2})}{\varepsilon(v_{2}%
)}-\frac{M^{\prime}(v_{2})}{\varepsilon^{\prime}(v_{2})}+\frac{M(v_{D-2}%
)}{\varepsilon(v_{D-2})}-\frac{M^{\prime}(v_{D-2})}{\varepsilon^{\prime
}(v_{D-2})}\leq\\
& \leq\frac{\left(  m_{1}-1\right)  \cdot M(v_{2})/m_{1}}{\varepsilon^{\prime
}(v_{2})}+\frac{\left(  m_{D-1}-m_{D-1}-m_{1}+1\right)  \cdot M(v_{D-1}%
)/m_{D-1}}{\varepsilon^{\prime}(v_{D-2})}=\\
& =\frac{\left(  m_{1}-1\right)  \cdot M(v_{2})/m_{1}}{\varepsilon^{\prime
}(v_{2})}-\frac{\left(  m_{1}-1\right)  \cdot M(v_{D-1})/m_{D-1}}%
{\varepsilon^{\prime}(v_{D-2})}%
\end{align*}
Since $\varepsilon^{\prime}(v_{2})\geq\varepsilon^{\prime}(v_{D-2})$ by
construction and since
\[
M(v_{2})/m_{1}\leq M(v_{D-1})/m_{D-1}%
\]
by assumption of the Theorem, we conclude $\xi^{ac}(T)-\xi^{ac}(T^{\prime
})\leq0.$ Since also obviously
\[
\frac{M(v_{0})}{\varepsilon(v_{0})}-\frac{M^{\prime}(v_{0})}{\varepsilon
^{\prime}(v_{0})}<0
\]
we obtain $\xi^{ac}(T)-\xi^{ac}(T^{\prime})<0.$
\end{proof}

\bigskip

Note that in every tree $T$ there must exist a diametric path which satisfies
conditions of Theorem \ref{tm_maxStar}, for either the condition holds for
diametric path $P$ or for the same path with vertices labeled in reverse
order. Applying this transformation repeatedly on a tree with diameter greater
than $5,$ we will finally obtain a tree of diameter $4$ with greater value of
augmented ECI.

\bigskip

Therefore, a tree with maximum value of augmented ECI lies among trees with
$D\leq4.$ Let us now consider such trees.

\begin{lemma}
\label{lemma_caterpilar}Let $T\not =S_{n}$ be a tree on $n$ vertices with
diameter $D\leq4$ such that central vertices have at most two non-leaf
neighbors. Then $\xi^{ac}(T)<\xi^{ac}(S_{n}).$
\end{lemma}

\begin{proof}
This lemma is corollary of Theorem \ref{tm_path}, since every tree $T$
satisfying conditions of this lemma can be obtained from $S_{n}$ by applying
once (if $D=3$) or twice (if $D=4$) the transformation of tree from that theorem.
\end{proof}

\bigskip

As a consequence of this lemma, we can conclude that the "problem" are trees
with diameter $4$ and at least three non-leaf neighbors of central vertex. Let
us now consider such trees. Before we proceed, let us note that tree
transformations from some of the following Lemmas are illustrated in Figure
\ref{Figure3}.%

%TCIMACRO{\FRAME{ftbpFU}{2.9871in}{2.4206in}{0pt}{\Qcb{Tree transformations
%from Lemmas \ref{lemma_balance}, \ref{lemma_deg2or3} and \ref{lemma_max(p=3)}
%respectively.}}{\Qlb{Figure3}}{figure3.eps}%
%width 2.9871in;  height 2.4206in;  depth 0pt;  original-width 19.7264in;
%original-height 15.9523in;  cropleft "0";  croptop "1";  cropright "1";
%cropbottom "0";  filename 'Figure3.eps';file-properties "XNPEU";}}}%
%BeginExpansion
\begin{figure}[ptb]%
\centering
\includegraphics[
height=2.4206in,
width=2.9871in
]%
{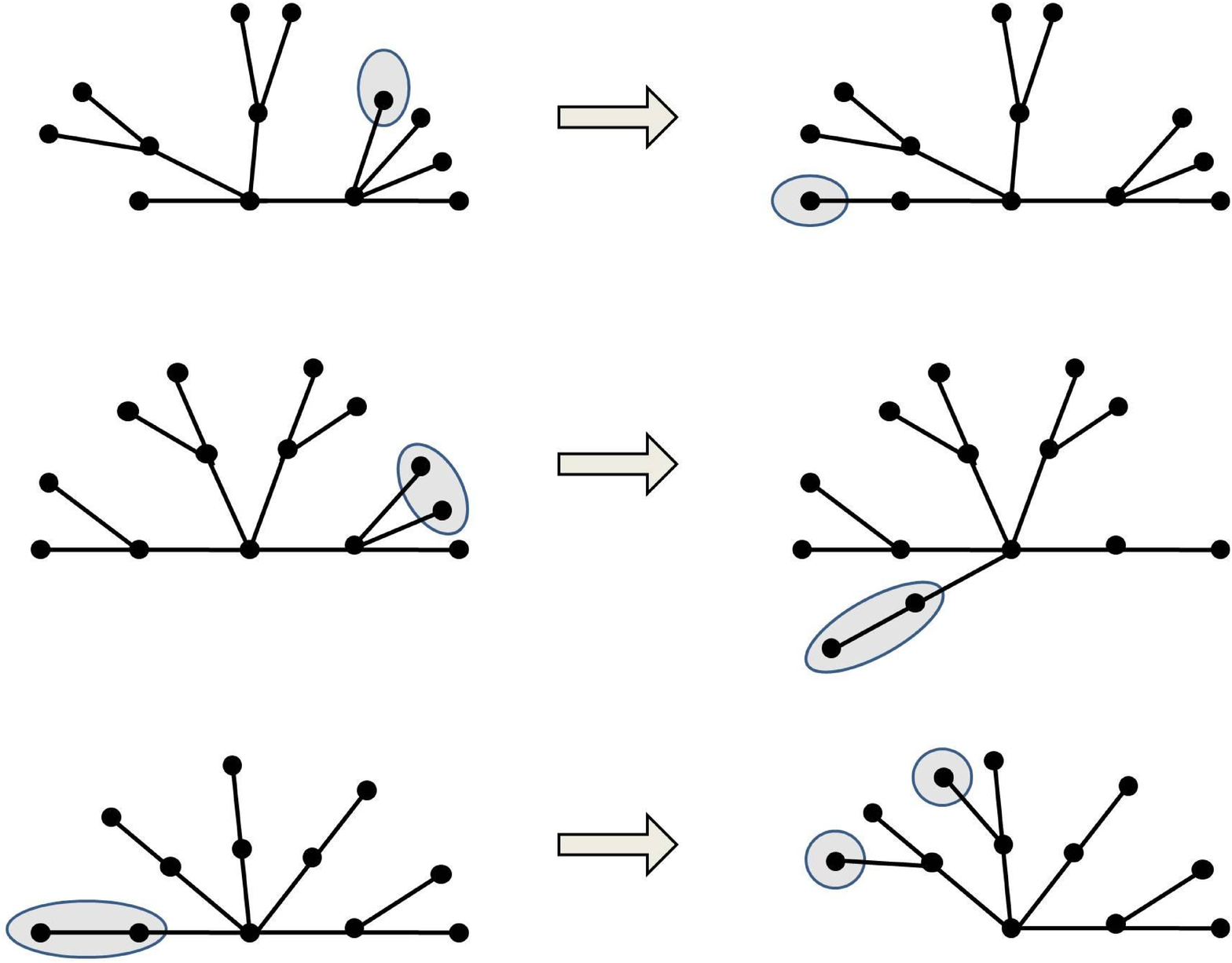}%
\caption{Tree transformations from Lemmas \ref{lemma_balance},
\ref{lemma_deg2or3} and \ref{lemma_max(p=3)} respectively.}%
\label{Figure3}%
\end{figure}
%EndExpansion

\begin{lemma}
\label{lemma_balance}Let $T$ be a tree on $n$ vertices with diameter $D=4$
such that central vertex $u$ has at least three non-leaf neighbors. Let
$v_{1},\ldots v_{k}$ be all neighbors of $u$ labeled so that $\deg(v_{1}%
)\leq\ldots\leq\deg(v_{k}).$ If $\deg(v_{k})-\deg(v_{1})\geq2$ than for a tree
$T^{\prime}$ obtained from $T$ by deleting a pending vertex of $v_{k}$ and
adding pending vertex to $v_{1}$ holds
\[
\xi^{ac}(T)<\xi^{ac}(T^{\prime}).
\]

\end{lemma}

\begin{proof}
Note that eccentricities of vertices remain the same after this
transformation. The only vertex whose degree decreases is $v_{k},$ therefore
$M(v)$ decreases possibly for $u$ and pending vertices of $v_{k}.$ Let us
denote $m_{i}=\deg v_{i}$ and $c=M(u)/\left(  m_{1}\cdot m_{k}\right)  .$ Note
that $c\geq2$ because central vertex has at least three non-leaf neighbors.
Considering vertex $u$ and pending vertices of $v_{1}$ and $v_{k}$ we obtain
\begin{align*}
\xi(T)-\xi(T^{\prime})  & =\frac{m_{1}\cdot c\cdot m_{k}}{2}-\frac{\left(
m_{1}+1\right)  \cdot c\cdot\left(  m_{k}-1\right)  }{2}+\left(
m_{1}-1\right)  \left(  \frac{m_{1}}{4}-\frac{m_{1}+1}{4}\right)  +\\
& +\left(  m_{k}-2\right)  \left(  \frac{m_{k}}{4}-\frac{m_{k}-1}{4}\right)
+\left(  \frac{m_{k}}{4}-\frac{m_{1}+1}{4}\right) \\
& =\frac{1}{2}\left(  c-1\right)  \left(  m_{1}-m_{k}+1\right)  \leq\frac
{1}{2}\left(  c-1\right)  \left(  -2+1\right)  <0.
\end{align*}

\end{proof}

\bigskip

Note that Lemma \ref{lemma_balance} holds even for trees with only two
non-pendant neighbors of central vertex. But then we do not necessarily have
strict inequality (constant $c$ from proof can be $1$). Lemma
\ref{lemma_balance} can be applied repeatedly until we obtain degree balanced
tree. Therefore, we conclude that among trees on $n$ vertices with $D=4$ and
central vertex $u$ with given degree $\deg(u)=m$ degree balanced tree
$TB_{n,m}$ has maximum value of $\xi^{ac}$. Now, we can obtain the increase in
$\xi^{ac}$ by changing the degree of central vertex. For that purpose we need
following lemma.

\begin{lemma}
\label{lemma_deg2or3}Let $T$ be a tree on $n$ vertices with diameter $D=4$
such that central vertex $u$ has at least three non-leaf neighbors. Let
$v_{1},\ldots v_{k}$ be all neighbors of $u$ labeled so that $\deg(v_{1}%
)\leq\ldots\leq\deg(v_{k}).$ If $\deg(v_{k})-\deg(v_{1})\leq1$ and $\deg
(v_{k})\geq4$ than for a tree $T^{\prime}$ obtained from $T$ by deleting two
pendant vertices of $v_{k}$ and adding pendant path of length $2 $ to $u$
holds
\[
\xi^{ac}(T)<\xi^{ac}(T^{\prime}).
\]

\end{lemma}

\begin{proof}
Let us denote $k=\deg(u)$, $m_{i}=\deg(v_{i})$ and $c=M(u)/m_{k}.$ Considering
$u,v_{1},\ldots,v_{k}$ and pendant vertices of $v_{k}$ we obtain
\begin{align*}
\xi(T)-\xi(T^{\prime})  & =\frac{c\cdot m_{k}}{2}-\frac{2\cdot c\cdot\left(
m_{k}-2\right)  }{2}+k\left(  \frac{k}{3}-\frac{k+1}{3}\right)  +\\
& +\left(  m_{k}-3\right)  \left(  \frac{m_{k}}{4}-\frac{m_{k}-2}{4}\right)
+\left(  2\cdot\frac{m_{k}}{4}-\left(  \frac{k+1}{3}+\frac{2}{4}\right)
\right) \\
& =-\frac{c}{2}\left(  m_{k}-4\right)  -\frac{2}{3}k+m_{k}-\frac{7}{3}%
\leq\lbrack c\geq6]\leq\\
& \leq\frac{29}{3}-2m_{k}-\frac{2}{3}k\leq\lbrack m_{k}\geq4,k\geq3]\leq
-\frac{1}{3}<0
\end{align*}
which concludes the proof.
\end{proof}

\bigskip

By combining Lemmas \ref{lemma_balance} and \ref{lemma_deg2or3},\ we can
conclude that we have restricted our search for trees with extremal $\xi^{ac}$
to $S_{n}$ or degree balanced trees $TB_{n,k}$ with
\[
p=\left\lceil \frac{n-1}{k}\right\rceil \leq3.
\]
Now we will consider separately cases when $p=2$ and $p=3.$ First we will
consider case when $p=3.$ For that purpose we need the following lemma.

\begin{lemma}
\label{lemma_max(p=3)}Let $T$ be a tree on $n$ vertices with diameter $D=4$
such that central vertex $u$ has at least three non-leaf neighbors. If
$T=TB_{n,k}$ where $p=\left\lceil \frac{n-1}{k}\right\rceil =3$ then
\[
\xi^{ac}(T)\leq\xi^{ac}(TB_{n,\left\lceil \frac{n-1}{3}\right\rceil })
\]
with equality if and only if $k=\left\lceil \frac{n-1}{3}\right\rceil .$
\end{lemma}

\begin{proof}
From $p=3$ follows that all neighbors of $u$ are of degree $3$ and possibly
$2.$ Suppose $p\not =\left\lceil \frac{n-1}{3}\right\rceil .$ That means there
are at least three neighbors of $u$ of degree $2.$ Let us denote all neighbors
of $u$ with $v_{1},\ldots,v_{k}$ so that $m_{1}\leq m_{2}\leq\ldots\leq m_{k}$
where $m_{i}$ denotes $\deg(v_{i}).$ Let $w_{1}$ be a pendant vertex of
$v_{1}$. Now, let $T^{\prime}$ be a tree obtained from $T$ by deleting edges
$uv_{1}$, $v_{1}w_{1}$ and adding edges $v_{2}v_{1},$ $v_{3}w_{1}.$ We will
show that $\xi^{ac}$ has increased by this transformation. For that purpose
let $c=M(u)/8.$ Considering vertices $u,v_{1},\ldots,v_{k}$ and pendant
vertices of $v_{1},v_{2}$ and $v_{3}$ we obtain
\begin{align*}
\xi^{ac}(T)-\xi^{ac}(T^{\prime})  & \leq\left(  \frac{c\cdot8}{2}-\frac
{c\cdot9}{2}\right)  +\left(  k-1\right)  \left(  \frac{k}{3}-\frac{k-1}%
{3}\right)  +\left(  \frac{k}{3}-\frac{3}{4}\right)  +3\left(  \frac{2}%
{4}-\frac{3}{4}\right) \\
& =\frac{2}{3}k-\frac{1}{2}c-\frac{11}{6}\leq\lbrack c\geq2^{k-3}]\leq\\
& \leq\frac{2}{3}k-2^{k-4}-\frac{11}{6}\leq\frac{2}{3}3-2^{3-4}-\frac{11}%
{6}=-\frac{1}{3}<0.
\end{align*}
Repeating this transformation, we obtain $TB_{n,\left\lceil \frac{n-1}%
{3}\right\rceil }$ which proves the lemma.
\end{proof}

\bigskip

Now, we want to address trees $TB_{n,k}$ with $p=\left\lceil \frac{n-1}%
{k}\right\rceil =2.$

\begin{lemma}
\label{lemma_max(p=2)}Let $T$ be a tree on $n$ vertices with diameter $D=4$
such that central vertex $u$ has at least three non-leaf neighbors. If
$T=TB_{n,k}$ where $p=\left\lceil \frac{n-1}{k}\right\rceil =2.$ Then%
\[
\xi^{\alpha c}(T)<\xi^{ac}(S_{n})\text{ \ \ or \ \ }\xi^{\alpha c}(T)<\xi
^{ac}(TB_{n,\left\lceil \frac{n-1}{3}\right\rceil })
\]

\end{lemma}

\begin{proof}
Note that it has to be $n\geq7$ for a tree $T$ to be able to satisfy
conditions of lemma. From $p=\left\lceil \frac{n-1}{k}\right\rceil =2$ follows
that neighbors of central vertex have degree $1$ or $2.$ Let $u$ be a central
vertex in $T$ and let $t$ neighbors of $u$ have degree $2.$ First note that
$2\leq t\leq\frac{n-1}{2}.$ Also, note that
\begin{align*}
n  & =1+2t+(k-t)=1+t+k,\\
k  & =n-t-1.
\end{align*}
Now we have
\[
\xi^{\alpha c}(TB_{n,k})=\frac{2^{t}}{2}+k\cdot\frac{k}{3}+t\cdot\frac{2}%
{4}=2^{t-1}+\frac{\left(  n-t-1\right)  ^{2}}{3}+\frac{t}{2}=f(t).
\]
Let us analyze obtained function $f(t).$ We have
\[
f^{\prime}(t)=2^{t-1}\cdot\ln2+\frac{2}{3}t+\frac{7}{6}-\frac{2}{3}n
\]
This is obviously increasing function in $k,$ therefore from
\begin{align*}
f^{\prime}(2)  & =2^{2-1}\cdot\ln2+\frac{2}{3}\cdot2+\frac{7}{6}-\frac{2}%
{3}n<0\text{ for }n\geq7.\\
f^{\prime}(\frac{n-1}{2})  & =2^{\frac{n-1}{2}-1}\cdot\ln2+\frac{2}{3}%
\cdot\frac{n-1}{2}+\frac{7}{6}-\frac{2}{3}n>0\text{ for }n\geq7,
\end{align*}
we conclude that maximum of $f(t)=\xi^{\alpha c}(TB_{n,k})$ is obtained for
$t=2$ or $t=\left\lfloor \frac{n-1}{2}\right\rfloor .$ If $t=2$ then by Lemma
\ref{lemma_caterpilar} we conclude that $\xi^{\alpha c}(T)<\xi^{ac}(S_{n}).$
If $t=\left\lfloor \frac{n-1}{2}\right\rfloor $ and $n$ is odd then by Lemma
\ref{lemma_caterpilar} we conclude that $\xi^{\alpha c}(T)<\xi^{ac}%
(TM_{n,3}).$ If $t=\left\lfloor \frac{n-1}{2}\right\rfloor $ and $n$ is odd,
then before we can apply Lemma \ref{lemma_caterpilar}, we have to prove that
tree $T$ with $t=\left\lfloor \frac{n-1}{2}\right\rfloor $ has smaller
$\xi^{ac}$ than tree $T^{\prime}$ obtained from it by deleting last pending
vertex of central vertex and adding one pending vertex to one neighbor of
central vertex. We have
\begin{align*}
\xi^{ac}(T)-\xi^{ac}(T^{\prime})  & =\left(  \frac{2^{t}}{2}+\frac{(t+1)^{2}%
}{3}+\frac{t}{2}\right)  -\left(  \frac{3\cdot2^{t-1}}{2}+\frac{t^{2}}%
{3}+\frac{t-1}{2}+2\cdot\frac{3}{4}\right)  =\\
& =\frac{2}{3}t-\frac{1}{4}2^{t}-\frac{2}{3}\leq-\frac{1}{3}<0
\end{align*}
which completes the proof.
\end{proof}

\bigskip

Therefore, a tree with extremal value of augmented ECI must be either $S_{n}$
or $TB_{n,\left\lceil \frac{n-1}{3}\right\rceil }.$ To decide which is it, we
need following lemma.

\begin{lemma}
Holds
\[
\xi^{ac}(TB_{n,\left\lceil \frac{n-1}{3}\right\rceil })>\xi^{ac}(S_{n}).
\]
if and only if $n\geq16.$
\end{lemma}

\begin{proof}
Let $T=TB_{n,\left\lceil \frac{n-1}{3}\right\rceil }$ and let $k=\left\lceil
\frac{n-1}{3}\right\rceil $ be the degree of central vertex in $T.$ From the
value of $k$ follows that neighbors of central vertex have degrees $3$ or $3$
and $2.$ Let $k_{2}$ be number of neighbors of central vertex with degree $2$
and let $k_{3}$ be number of neighbors of central vertex with degree $3.$
Obviously $k_{2}\leq2$ and $k_{2}+k_{3}=k.$ We have
\[
n=1+k_{2}+k_{3}+k_{2}+2k_{3}=1+3k-k_{2}.
\]
We distinguish three cases. Given the exact formula from Proposition
\ref{prop_values} we have
\[
\xi^{ac}(T)-\xi^{ac}(S_{n})=\left\{
\begin{tabular}
[c]{ll}%
$\frac{1}{2}3^{k}-\left(  \frac{25}{6}k^{2}-\frac{3}{2}k+1\right)  $ & for
$n=3k+1,$\\
$3^{k-1}-\left(  \frac{25}{6}k^{2}-\frac{9}{2}k+\frac{5}{2}\right)  $ & for
$n=3k,$\\
$\frac{2}{9}3^{k}-\left(  \frac{25}{6}k^{2}-\frac{15}{2}k+\frac{7}{2}\right)
$ & for $n=3k-1.$%
\end{tabular}
\right.
\]
Since exponential function grows faster then polynomial, from analyzing
obtained expressions we conclude that difference $\xi^{ac}(T)-\xi^{ac}(S_{n})$
is strictly positive:

\begin{itemize}
\item for $k\geq5$ $(n\geq16)$ in the case of $n=3k+1$,

\item for $k\geq6$ $(n\geq18)$ in the case od $n=3k$,

\item for $k\geq6$ $(n\geq17)$ in the case of $n=3k-1$.
\end{itemize}

\noindent Otherwise, the difference is strictly negative. That concludes the proof.
\end{proof}

\bigskip

Now, we can summarize our results in the following theorem.

\begin{theorem}
\label{tm_maxTrees}Let $T$ be a tree on $n$ vertices. Then
\[
\xi^{ac}(T)\leq\left\{
\begin{tabular}
[c]{ll}%
$\xi^{ac}(S_{n})$ & if $n\leq15,$\\
$\xi^{ac}(TB_{n,\left\lceil \frac{n-1}{3}\right\rceil })$ & if $n\geq16,$%
\end{tabular}
\right.
\]
with equality holding id and only if $T=S_{n}$ for $n\leq15$ and $T=\xi
^{ac}(TB_{n,\left\lceil \frac{n-1}{3}\right\rceil })$ for $n\geq16.$
\end{theorem}

\section{Extremal trees with perfect matching}

In this section we assume $n$ to be even since only trees with even $n$ can
have perfect matching. Also, tree with a perfect matching can obviously have
at most one pendant vertex on every vertex in it. Furthermore, if
$P=v_{0}v_{1}\ldots v_{D}$ is diametric path in a tree with a perfect matching
then $v_{1}$ and $v_{D-1}$ must be of degree $2$ since they already have one
pendent vertex and can't have more.

Path $P_{n}$ on $n$ vertices obviously has perfect matching, therefore the
problem of finding a tree with perfect matching and minimum $\xi^{ac}$ is
trivial - it is $P_{n}.$ But star $S_{n}$ and degree balanced tree
$TB_{n,\left\lceil \frac{n-1}{3}\right\rceil }$ do not have a perfect
matching. Therefore finding a tree with perfect matching and maximum $\xi
^{ac}$ is a nontrivial one. In order to find such tree, we will introduce
transformation of a tree which preserves existence of perfect matching and
increases $\xi^{ac}$. But before that, note that for every integer $m\geq2$
the following inequalities hold
\begin{align}
m  & \leq2^{m-1},\label{for_matching1}\\
2^{m-2}-m+1  & \geq0.\label{for_matching2}%
\end{align}
Also, for every integer $m\geq4$ holds
\begin{equation}
m\leq2^{m-2}.\label{for_matching3}%
\end{equation}
Now, we can state the following theorem.

\begin{theorem}
\label{tm_match}Let $T$ be a tree on $n$ vertices with a perfect matching and
$D\geq5$. Let $P=v_{0}v_{1}\ldots v_{D}$ be a diametric path in $T$ and let
$w_{1},\ldots,w_{k}$ be all vertices from $V\backslash\left\{  v_{3}\right\}
$ adjacent to $v_{2}$ and of degree $2.$ Let $T^{\prime}$ be a tree obtained
from $T$ by deleting edges $v_{2}w_{1},\ldots,v_{2}w_{k}$ and adding edges
$v_{3}w_{1},\ldots,v_{3}w_{k}.$ Then $T^{\prime}$ is a tree on $n$ vertices
which has perfect matching and
\[
\xi^{ac}(T)<\xi^{ac}(T^{\prime}).
\]

\end{theorem}

\begin{proof}
First note that $k\geq1$ since at least $v_{1}$ is included among
$w_{1},\ldots,w_{k}.$ Let $m_{i}$ denote a degree of $v_{i}\in P.$ Note that
$m_{2}-2\leq k\leq m_{2}-1$ since all neighbors of $v_{2}$ are of degree $2$
except possibly $v_{3}$ (which is not counted in $k$ by construction) and one
pendant vertex. Now, eccentricities of all vertices do not increase by this
transformation. The only vertices whose $M(v)$ possibly decreases are $v_{2},$
$v_{3}$ and a pendent vertex of $v_{2}$ (if such exists). We distinguish three cases.

Case 1. $v_{2}$ has one pendent vertex and $v_{3}$ has one pendent vertex.

Let us denote with $u_{2}$ and $u_{3}$ pendant vertices of $v_{2}$ and $v_{3}
$ respectively. In this case $k=m_{2}-2.$ We have
\begin{align*}
\Delta_{1}  & =\frac{M(v_{2})}{\varepsilon(v_{2})}-\frac{M^{\prime}(v_{2}%
)}{\varepsilon^{\prime}(v_{2})}+\frac{M(v_{3})}{\varepsilon(v_{3})}%
-\frac{M^{\prime}(v_{3})}{\varepsilon^{\prime}(v_{3})}\leq\\
& \leq\frac{2^{m_{2}-2}\cdot m_{3}}{\varepsilon^{\prime}(v_{2})}-\frac
{m_{3}+m_{2}-2}{\varepsilon^{\prime}(v_{2})}+\frac{m_{2}\cdot2^{m_{3}-3}\cdot
m_{4}}{\varepsilon^{\prime}(v_{3})}-\frac{2\cdot2^{m_{3}-3}\cdot2^{m_{2}%
-2}\cdot m_{4}}{\varepsilon^{\prime}(v_{3})}%
\end{align*}
Since $\varepsilon^{\prime}(v_{2})>\varepsilon^{\prime}(v_{3})$ and $m_{4}%
\geq2,$ in order to prove that $\Delta_{1}>0$ it is sufficient to prove that
\[
m_{3}\cdot(2^{m_{2}-2}-1)-m_{2}+2\leq2^{m_{3}-3}\cdot2\cdot\left(
2\cdot2^{m_{2}-2}-m_{2}\right)  .
\]
If $m_{3}=3$, then this inequality becomes $2^{m_{2}-2}-m_{2}+1\geq0$ which is
actually inequality (\ref{for_matching2}) for $m_{2}\geq2$ and therefore
holds. If $m_{3}\geq4,$ then since $m_{2}-2\geq0,$ it is sufficient to prove
\[
m_{3}\cdot(2^{m_{2}-2}-1)\leq2^{m_{3}-3}\cdot2\cdot\left(  2\cdot2^{m_{2}%
-2}-m_{2}\right)
\]
which follows from (\ref{for_matching3}) for $m_{3}\geq4$ and
(\ref{for_matching2}) for $m_{2}\geq2.$ Also, from $\varepsilon^{\prime}%
(u_{2})>\varepsilon^{\prime}(u_{3})$ follows that
\[
\Delta_{2}=\frac{M(u_{2})}{\varepsilon(u_{2})}-\frac{M^{\prime}(u_{2}%
)}{\varepsilon^{\prime}(u_{2})}+\frac{M(u_{3})}{\varepsilon(u_{3})}%
-\frac{M^{\prime}(u_{3})}{\varepsilon^{\prime}(u_{3})}\leq\frac{m_{2}%
}{\varepsilon^{\prime}(u_{2})}-\frac{2}{\varepsilon^{\prime}(u_{2})}%
+\frac{m_{3}}{\varepsilon^{\prime}(u_{3})}-\frac{m_{3}+m_{2}-2}{\varepsilon
^{\prime}(u_{3})}<0.
\]
Now
\[
\xi^{ac}(T)-\xi^{ac}(T^{\prime})\leq\Delta_{1}+\Delta_{2}<0.
\]

Case 2. $v_{2}$ has one pendent vertex and $v_{3}$ has no pendant vertices.

Let us denote with $u_{2}$ pendant vertex of $v_{2}$. In this case
$k=m_{2}-2.$ We have
\begin{align*}
\xi^{ac}(T)-\xi^{ac}(T^{\prime})  & \leq\frac{M(v_{2})}{\varepsilon(v_{2}%
)}-\frac{M^{\prime}(v_{2})}{\varepsilon^{\prime}(v_{2})}+\frac{M(v_{3}%
)}{\varepsilon(v_{3})}-\frac{M^{\prime}(v_{3})}{\varepsilon^{\prime}(v_{3}%
)}+\frac{M(u_{2})}{\varepsilon(u_{2})}-\frac{M^{\prime}(u_{2})}{\varepsilon
^{\prime}(u_{2})}\leq\\
& \leq\frac{m_{3}\cdot\left(  2^{m_{2}-2}-1\right)  }{\varepsilon^{\prime
}(v_{2})}-\frac{m_{2}-2}{\varepsilon^{\prime}(v_{2})}-\frac{2^{m_{3}-2}\cdot
m_{4}\cdot\left(  2\cdot2^{m_{2}-2}-m_{2}\right)  }{\varepsilon^{\prime}%
(v_{3})}+\frac{m_{2}-2}{\varepsilon^{\prime}(u_{2})}%
\end{align*}
Since $\varepsilon^{\prime}(u_{2})>\varepsilon^{\prime}(v_{2})>\varepsilon
^{\prime}(v_{3})$ and $m_{4}\geq2,$ it is sufficient to prove that
\[
m_{3}\cdot\left(  2^{m_{2}-2}-1\right)  \leq2^{m_{3}-2}\cdot2\cdot\left(
2\cdot2^{m_{2}-2}-m_{2}\right)  .
\]
But this follows from inequality (\ref{for_matching1}) for $m_{3}\geq2$ and
(\ref{for_matching2}) for $m_{2}\geq2.$

Case 3. $v_{2}$ has no pendent vertices.

Note that in this case edge $v_{2}v_{3}$ must be included in a perfect
matching, therefore $v_{3}$ cannot have pendent vertices. In this case
$k=m_{2}-1.$ Note that $\varepsilon^{\prime}(w_{i})=\varepsilon^{\prime}%
(v_{2}).$ We have
\begin{align*}
\xi^{ac}(T)-\xi^{ac}(T^{\prime})  & \leq\frac{M(v_{2})}{\varepsilon(v_{2}%
)}-\frac{M^{\prime}(v_{2})}{\varepsilon^{\prime}(v_{2})}+\frac{M(v_{3}%
)}{\varepsilon(v_{3})}-\frac{M^{\prime}(v_{3})}{\varepsilon^{\prime}(v_{3}%
)}+\sum_{i=1}^{k}\left(  \frac{M(w_{i})}{\varepsilon(w_{i})}-\frac{M^{\prime
}(w_{i})}{\varepsilon^{\prime}(w_{i})}\right)  \leq\\
& \leq\frac{m_{3}\cdot\left(  2^{m_{2}-1}-1\right)  -m_{2}+1}{\varepsilon
^{\prime}(v_{2})}-\frac{2^{m_{3}-2}\cdot m_{4}\cdot\left(  2\cdot2^{m_{2}%
-1}-m_{2}\right)  }{\varepsilon^{\prime}(v_{3})}+\\
& +\frac{\left(  m_{2}-1\right)  \left(  1-m_{3}\right)  }{\varepsilon
^{\prime}(v_{2})}%
\end{align*}
Since $\varepsilon^{\prime}(v_{2})>\varepsilon^{\prime}(v_{3})$ and $m_{4}%
\geq2$ it is sufficient to prove that
\[
m_{3}\cdot\left(  2^{m_{2}-1}-1\right)  -m_{2}+1+\left(  m_{2}-1\right)
\left(  1-m_{3}\right)  \leq2^{m_{3}-2}\cdot2\cdot\left(  2\cdot2^{m_{2}%
-1}-m_{2}\right)  ,
\]
which is equivalent to
\[
m_{3}\cdot\left(  2^{m_{2}-1}-m_{2}\right)  \leq2^{m_{3}-1}\left(
2\cdot2^{m_{2}-1}-m_{2}\right)
\]
and follows from (\ref{for_matching1}) for $m_{3}\geq2.$
\end{proof}

\bigskip

Now, as a corollary to this theorem we obtain the only extremal tree with
respect to augmented ECI among trees with perfect matching.

\begin{corollary}
\label{cor_maxMatch}Let $T\not =TB_{n,\frac{n}{2}}$ be a tree on $n$ vertices
with perfect matching. Then
\[
\xi^{ac}(T)<\xi^{ac}(TB_{n,\frac{n}{2}}).
\]

\end{corollary}

\begin{proof}
We apply the transformation from Theorem \ref{tm_match} on $T.$ Note that each
transformation decreases diameter by $1$ until finally we obtain the tree of
diameter $4.$ The only tree of diameter $4$ which has perfect matching is
$TB_{n,\frac{n}{2}}.$
\end{proof}

\section{Extremal graphs}

Let us now establish extremal graphs among all simple connected graphs. Those
results will follow easily from results for threes. First, the following
proposition obviously holds, since contribution of every vertex to $\xi^{ac}$
in complete graph $K_{n}$ is maximum possible.

\begin{proposition}
\label{prop_maxGraphs}For a graph $G\not =K_{n}$ on $n$ vertices holds
\[
\xi^{ac}(G)<\xi^{ac}(K_{n}).
\]

\end{proposition}

Therefore, we have established only maximal graphs with respect to the value
of augmented ECI. In the following proposition we establish minimal graphs.

\begin{proposition}
\label{prop_minGraphs}For a graph $G\not =P_{n}$ on $n$ vertices holds
\[
\xi^{ac}(G)>\xi^{ac}(P_{n}).
\]

\end{proposition}

\begin{proof}
Let $T$ be spanning tree of $G.$ From definition of spanning tree follows that
$T$ is obtained from $G$ by deleting some edges. Note that deleting edges does
not decrease eccentricities of vertices. If $G$ is already a tree, then the
result follows from Corollary \ref{Cor_path}. If $G$ is not a tree, then we
have to delete at least one edge. Note that by deleting edges degrees of
vertices (and therefore values $M(v)$) do not increase. Since we deleted at
least one edge, that means that the degree of at least one vertex decreased
and we have
\[
\xi^{ac}(G)>\xi^{ac}(T)\geq\xi^{ac}(P_{n})
\]
which concludes the proof.
\end{proof}

\section{Conclusion}

In this paper we studied augmented eccentric connectivity index on graphs and
trees. We established that minimal trees with respect to augmented ECI are
paths $P_{n}$ (Corollary \ref{Cor_path}), while maximal trees are either stars
$S_{n}$ for $n\leq15$ either degree balanced trees $TB_{n,\left\lceil
\frac{n-1}{3}\right\rceil }$ for $n\geq16$ (Theorem \ref{tm_maxTrees}). Using
similar techniques we proved that in the class of trees with perfect matching
minimal trees are again paths $P_{n}$, while maximal trees are $TB_{n,\frac
{n}{2}}$ (Corollary \ref{cor_maxMatch}). In the class of general simple
connected graphs on $n$ vertices, maximal graphs with respect to augmented ECI
are complete graphs $K_{n}$ (Proposition \ref{prop_maxGraphs}), while minimal
graphs are paths $P_{n}$ (Proposition \ref{prop_minGraphs}). The explicit
formulas for the values of augmented ECI of all these graphs which turned out
to be extremal are derived and presented in Proposition \ref{prop_values}.

There are many open questions for further study. In this paper we only
initiated studying extremal trees with given parameter (trees with perfect
matching). One could try to establish extremal trees with given diameter,
radius, number of pendant vertices, maximum degree (chemical trees), etc.
Also, one could try to establish extremal unicyclic graphs with respect to
augmented ECI. Deriving exact formulas for the value of augmented ECI on some
special kinds of graphs would also be interesting, just as studying of how
augmented ECI behaves with respect to graph operations.

Finally, there is also super augmented ECI, which is similar to augmented ECI,
and is defined with
\[
\xi^{ac}(G)=\sum_{u\in V}\frac{M(u)}{\varepsilon^{2}(u)}.
\]
It would be interesting to derive all those results for that index too. As for
the results from this paper, we mostly relied on order of $\varepsilon(v)$
between pairs of vertices. Since the same order holds for $\varepsilon^{2}(v)$
then the results for $\xi^{sac}(G)$ should be perfectly analogous.

\section{Acknowledgements}

Partial support of the Ministry of Science, Education and Sport of the
Republic of Croatia (grant. no. 083-0831510-1511) and of project Gregas is
gratefully acknowledged.

\end{document}